\documentclass[11pt, oneside, reqno]{amsart}

\newtheorem{thm}{Theorem}[section]
\newtheorem{lem}[thm]{Lemma}
\newtheorem{prop}[thm]{Proposition}
\newtheorem{cor}[thm]{Corollary}

\theoremstyle{definition}
\newtheorem{defn}[thm]{Definition}
\newtheorem{examp}[thm]{Example}

\newtheorem{notn}[thm]{Notation}

\theoremstyle{remark}
\newtheorem{rmk}[thm]{Remark}

\numberwithin{equation}{section}
\usepackage{latexsym}
\usepackage{amsmath}
\usepackage{amsthm}
\usepackage{amssymb}
\usepackage{amsfonts}
\usepackage[all]{xy}
\usepackage{graphicx}
\usepackage{lscape}
\usepackage{verbatim}
\usepackage{wrapfig}
\usepackage{subfig}
\usepackage{setspace}
\usepackage{pinlabel}
\usepackage{enumerate, setspace, dsfont}
\usepackage{fancyhdr}
\usepackage{calc}
\usepackage{rotating}

\begin{document}

\title{The Moduli Space of Hex Spheres}
\author{Aldo-Hilario Cruz-Cota}
\address{Department of Mathematics, Grand Valley State University, Allendale, MI 49401-9401, USA}
\email{cruzal@gvsu.edu}



\date{\today}

\dedicatory{}

\begin{abstract}
\noindent A \emph{hex sphere} is a singular Euclidean sphere with four cone points whose cone angles are (integer) multiples of $\frac{2\pi}{3}$ but less than $2\pi$. We prove that the Moduli space of hex spheres of unit area is homeomorphic to the the space of similarity classes of Voronoi polygons in the Euclidean plane. This result gives us as a corollary that each unit-area hex sphere $M$ satisfies the following properties: ($1$) it has an embedded (open Euclidean) annulus  that is disjoint from the singular locus of $M$; ($2$) it embeds isometrically in the 3-dimensional Euclidean space as the boundary of a tetrahedron; and ($3$) there is a simple closed geodesic $\gamma$ in $M$ such that a fractional Dehn twist along $\gamma$ converts $M$ to the double of a parallelogram.
\end{abstract}

\maketitle

\section{Introduction} \label{sect-intro}
A surface is \emph{singular Euclidean} if it is locally modeled on either the Euclidean plane or a Euclidean cone. In this article we study a special type of singular Euclidean spheres that we call \emph{hex spheres}. These are defined as singular Euclidean spheres with four cone points which have cone angles that are multiples of $\frac{2\pi}{3}$ but less than $2\pi$. Singular Euclidean surfaces whose cone angles are multiples of $\frac{2\pi}{3}$ are mainly studied because they arise as limits at infinity of real projective structures. 

We now give examples of hex spheres. Consider a parallelogram $P$ on the Euclidean plane such that two of its interior angles equal $\pi/3$, while the other two equal $2\pi/3$. Such a parallelogram will be called a \emph{perfect parallelogram}. The  double $D$ of a perfect parallelogram $P$ is an example of a hex sphere. This example gives rise to a $3$-parameter family of hex spheres. To see this, let $\gamma$ be the simple closed geodesic in $D$ that is the double of a segment in $P$  that is perpendicular to one of the longest sides of $P$. Then two parameters of the family of hex spheres correspond to the lengths of two adjacent sides of $P$, and the other parameter corresponds to twisting $D$ along $\gamma$.

Let $M$ be a hex sphere. The Gauss-Bonnet Theorem implies that exactly two of the cone angles  of $M$ are equal to $\frac{4\pi}{3}$, while the other two are equal to $\frac{2\pi}{3}$. We consider the Voronoi decomposition of $M$ centered at the two cone points of angle $\frac{4\pi}{3}$. This decomposes $M$ into two cells, the Voronoi cells, which intersect along a graph in $M$, the Voronoi graph. Each Voronoi cell embeds isometrically in a Euclidean cone as a convex geodesic polygon (see \cite{vor-graphs}). By cutting a Voronoi cell along a shortest geodesic passing through the vertex of the cone, the Voronoi cell becomes a polygon on the Euclidean plane. This polygon will be called a \emph{planar Voronoi polygon}.

Our main result is the following:

\begin{thm}
The space of similarity classes of Voronoi polygons in the Euclidean plane is a parameter space for the Moduli space of hex spheres of unit area.
\end{thm}

Each point in the space of similarity classes of Voronoi polygons is uniquely determined by a pair of angles satisfying certain linear inequalities. Therefore, the main result implies that, roughly speaking, the isometry class of a hex sphere of unit area is determined by a pair of angles.

\begin{cor}
 Every unit-area hex sphere satisfies the following properties: 
\begin{itemize}
 \item it has an embedded, totally geodesic Euclidean annulus of positive width (in particular, every hex sphere has a simple closed geodesic);
\item it embeds isometrically in the 3-dimensional Euclidean space as the boundary of a tetrahedron;
\item there is a simple closed geodesic $\gamma$ in $M$ such that a fractional Dehn twist along $\gamma$ converts $M$ to the double of a perfect parallelogram.
\end{itemize}
\end{cor}

The author would like to thank his PhD adviser, Daryl Cooper, for many helpful discussions. Portions of this work were completed at the University of California, Santa Barbara and Grand Valley State University.

\section{Singular Euclidean Surfaces} \label{sect-SES}
\begin{defn} (\cite{Troyanov1})  \label{defn_cone_manif}
A closed triangulated surface $M$ is \emph{singular Euclidean} if it satisfies the following properties:

\begin{enumerate}
\item For every 2-simplex $T$ of $M$ there is a simplicial homeomorphism $f_T$ of $T$ onto a non-degenerate triangle $f_T(T)$ in the Euclidean plane.
\item If $T_1$ and $T_2$ are two 2-simplices of $M$ with non-empty intersection, then there is an isometry $g_{12}$ of the Euclidean plane such that $f_{T_1}=g_{12}f_{T_2}$ on $T_1 \cap T_2$.
\end{enumerate}
\end{defn}

There is a natural way to measure the length of a curve $\gamma$ in a singular Euclidean surface $M$. This notion of length of curves coincides with the Euclidean length on each triangle of $M$ and it turns $M$ into a path metric space. That is, there is a \emph{distance function} $d_M$ on $M$ for which the distance between two points in $M$ is the infimum of the lengths of the paths in $M$ joining these two points. There is also a natural way to define an area measure in a singular Euclidean surface (\cite{Troyanov1}). This measure coincides with the usual Lebesgue measure on each Euclidean triangle of the surface. 

\begin{defn}
Let $M$ be a singular Euclidean surface $M$ and let $p$ be a point in $M$. The \emph{cone angle} of $M$ at $p$ is either $2\pi$ (if $p$ is not a vertex of $M$) or the sum of the angles of all triangles in $M$ that are incident to $p$ (if $p$ is a vertex of $M$). If $\theta$ is the cone angle of $M$ at $p$, then the number $k=2\pi-\theta$ is the (concentrated) \emph{curvature} of $M$ at $p$.
\end{defn}
The next definition generalizes the concept of tangent plane (see \cite{Burago}).
\begin{defn} (\cite{Cooper})
 Given a singular Euclidean surface $M$ and a point $p \in M$ of cone angle $\theta$, the \emph{tangent cone} $T_pM$ of $M$ at $p$ is the union of the Euclidean tangent cones to all the 2-simplices containing $p$.  The cone $T_pM$ is isometric to a Euclidean cone of angle $\theta$.
\end{defn}
A point $p$ in a singular Euclidean surface $M$ is called \emph{regular} if its cone angle equals $2\pi$. Otherwise it is called a \emph{singular} point or a \emph{cone point}. 
The \emph{singular locus} $\Sigma$ is the set of all singular points in $M$.

\section{Hex Spheres}

\begin{defn} \label{hex_spher}
A \emph{hex sphere} is an oriented singular Euclidean sphere with $4$ cone points whose cone angles are integer multiples of $\frac{2\pi}{3}$ but less than $2\pi$.
\end{defn}

Examples of hex spheres are given in the introduction of this paper.

\noindent $\diamond$ \emph{Why cone angles that are multiples of $\frac{2\pi}{3}$?}
Singular Euclidean surfaces with these cone angles arise naturally as limits at infinity of real projective structures. Real projective structures have been studied extensively by many authors (\cite{Goldman}, \cite{Choi-Goldman}, \cite{Loftin}, \cite{Labourie}, \cite{Hitchin}).

\noindent $\diamond$ \emph{Why $4$ cone points?} The following lemma shows that there is only one singular Euclidean sphere with $3$ cone points whose cone angles satisfy the numeric restrictions we are interested in. This suggests studying the next simplest case (when the singular sphere has $4$ cone points).

\begin{lem} \label{lem-hex-3-cone}
Let $M$ be a singular Euclidean sphere with $k$ singular points and assume that each cone angle of $M$ is an integer multiple of $\frac{2\pi}{3}$. Then $k \geq 3$, and if $k=3$ then $M$ is the double of an Euclidean equilateral triangle.
\end{lem}

The proof of Lemma \ref{lem-hex-3-cone} follows from the Gauss-Bonnet Theorem and Proposition 4.4 from \cite{Cooper}. The Gauss-Bonnet theorem also tells us the sizes of the cone angles of a hex sphere:

\begin{lem} \label{sizes_con-angles}
Exactly two of the cone angles of a hex sphere equal $\frac{2\pi}{3}$ while the other two equal $\frac{4\pi}{3}$.
\end{lem}

From now on, we will use the following notation:

\begin{itemize}
\item $M$ will denote a hex sphere.
\item $a$ and $b$ will denote the two cone points in $M$ of angle $\frac{4\pi}{3}$.
\item $c$ and $d$ will denote the two cone points in $M$ of angle $\frac{2\pi}{3}$.
\item $\Sigma=\{a,b,c,d\}$ will denote the singular locus of $M$.
\end{itemize}

\begin{defn} \label{defn-vor-reg}
 The (open) \emph{Voronoi region} $Vor(a)$ centered at $a$ is the set of points in $M$ consisting of:
\begin{itemize}
 \item the cone point $a$, and
 \item all non-singular points $x$ in $M$ such that 
  \begin{enumerate}
   \item $d_M(a,x) < d_M(b,x)$ and
   \item there exists a \emph{unique} shortest geodesic from $x$ to $a$.
  \end{enumerate}
\end{itemize}
\end{defn}

The (open) Voronoi region $Vor(b)$ centered at $b$ is defined by swapping the roles of $a$ and $b$ in Definition \ref{defn-vor-reg}. It is shown in \cite{Boileau} and \cite{Cooper} that, if $p \in \{a,b\}$, then there is a natural isometric embedding $f_p$ of $Vor(p)$ into the cone $T_pM$. We will use the following notation:

\begin{itemize}
 \item $A$ will be the closure of $f_a(Vor(a))$ in $T_aM$.
 \item $B$ will be the closure of $f_b(Vor(b))$ in $T_bM$.
 \item $\sqcup$ denotes the disjoint union of sets.
\end{itemize}

\section{The Moduli Space of Hex Spheres}

\begin{defn} \label{top_hex_spher}
 A \emph{topological hex sphere} is an oriented triangulated sphere $S$ together with $4$ pairs $(x_1,\theta_1)$, $(x_2,\theta_2)$, $(x_3,\theta_3)$, $(x_4,\theta_4)$ such that:
\begin{itemize}
\item The $x_i$'s are distinct points in the surface $S$.
\item Each $\theta_i>0$ is an integer multiple of $2\pi/3$ but less than $2\pi$.
\item $\sum^{4}_{i=1}(2\pi-\theta_{i})=2\pi\chi(S)$. [$\chi(S)$ is the Euler characteristic of $S$.]
\end{itemize}
\end{defn}

The point $x_i$ will be called a ``cone point'' in $S$ of ``cone angle'' $\theta_i$.

\begin{defn}
A \emph{hex structure} marked by an topological hex sphere $S$ is a pair $(M,h)$, which consists of a hex sphere $M$ together with an orientation-preserving piecewise-linear homeomorphism $h \colon S \to M$. We also require that the homeomorphism $h$ sends a cone point of $S$ to a  cone point of $M$ with the same cone angle.
\end{defn}

The set of all hex structures marked by $S$ will be denoted by $\mathcal{H}(S)$.

\begin{defn}
We define the equivalence relation $\simeq_\mathcal{M}$ on $\mathcal{H}(S)$ as follows: $(M,h) \simeq_\mathcal{M} (M',h')$ if and only if there is an orientation-preserving isometry $g \colon M \to M'$ which preserves the cone points (i.e., if the $x_i$'s are the cone points of $S$, then $g(h(x_i))=h'(x_i)$ for each $i$). The (oriented) \emph{moduli space} $\mathcal{M}(S)$ of $S$ is defined as the set $\mathcal{H}(S)$ modulo the equivalence relation $\simeq_\mathcal{M}$.
\end{defn}

The $\simeq_\mathcal{M}$ equivalence class of $(M,h)$ will be denoted by $[M,h]_{\simeq_\mathcal{M}} \in \mathcal{M}(S)$.

\begin{defn} \label{defn_metric_Mod}
Let $\textbf{d}_{\mathcal{M}(S)}$ be the map of $\mathcal{M}(S) \times \mathcal{M}(S)$ into the non-negative reals defined by $\textbf{d}_{\mathcal{M}(S)}([M,h]_{\simeq_\mathcal{M}},[M',h']_{\simeq_\mathcal{M}})=
\log \inf K$. The infimum is taken over all numbers $K \geq 1$ such that there is an orientation-preserving PL $K$-bi-Lipschitz homeomorphism $g \colon M \to M'$ preserving the cone points of $S$ and the Voronoi graphs of the hex spheres. 
\end{defn}

The reader can check that $\textbf{d}_{\mathcal{M}(S)}$ is a metric on $\mathcal{M}(S)$.

\section{Special Polygons} \label{sec-ang-param-spec-polyg}

Cutting each of the Voronoi cells $A \subset T_aM$ and $B \subset T_bM$ along a shortest geodesic passing through the vertex of the cone it lies on, we obtain the planar \emph{Voronoi polygons} $P_A$ and $P_B$ of $M$. These polygons satisfy certain properties, depending on the number $p$ of edges of $A$ ($p$ can only be equal to $2$, $3$ or $4$, see \cite{vor-graphs}).

\begin{defn} \label{defn_spec_polyg_p2}
A planar polygon $P$ is \emph{$2$-special} if it satisfies the following:
\begin{enumerate}
 \item $P$ has $4$ sides.
 \item There is exactly one vertex $a$ of $P$ at which the corner angle equals $4\pi/3$.
 \item The corner angle of $P$ at a vertex different than $a$ is less than  $\pi$.
 \item The corner angle of $P$ at the vertex opposite to $a$ equals $\pi/3$.
 \item Vertices of $P$ that are different than $a$ are equidistant from it.
 \end{enumerate}
\end{defn}

We reformulate the main result of \cite{vor-graphs} for the case $p=2$ as follows.

\begin{thm} \label{thm_spec_polyg_p2}
If $p=2$, then the (isometric) planar polygons $P_A$ and $P_B$ are $2$-special. Furthermore, the hex sphere $M$ can be recovered from the disjoint union of the polygons $P_A$ and $P_B$ by identifying the edges on the boundaries  of $P_A$ and $P_B$ as shown in Figure \ref{fig_spec_polyg_p2}. 
\end{thm}

\begin{figure}[ht!]
\labellist
\small\hair 3pt

\pinlabel $a$ [b] at 87 603
\pinlabel $\frac{4\pi}{3}$ [t] at 86 588 
\pinlabel $\frac{\pi}{3}$ [b] at 95 436
\pinlabel $P_A$ at 91 682

\pinlabel $b$ [b] at 497 603
\pinlabel $\frac{4\pi}{3}$ [t] at 495 588 
\pinlabel $\frac{\pi}{3}$ [b] at 504 436
\pinlabel $P_B$ at 500 682
\endlabellist
\centering
  \includegraphics[width=0.5\textwidth]{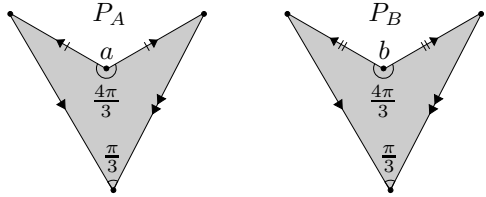}
  \caption{Recovering $M$ from $P_A \sqcup P_B$  when $p=2$}
  \label{fig_spec_polyg_p2}
\end{figure}

\begin{defn} \label{defn_spec_polyg_p3}
A planar polygon $P$ is \emph{$3$-special} if it satisfies the following:
\begin{enumerate}
 \item $P$ has $5$ sides.
 \item There is exactly one vertex $a$ of $P$ at which the corner angle equals $4\pi/3$.
 \item The corner angle of $P$ at a vertex different than $a$ is less than  $\pi$.
 \item The corner angle of $P$ at each vertex adjacent to $a$ equals $\pi/3$.
 \item The two sides of $P$ that have $a$ as an endpoint have the same length. Let $ac'$ and $ac''$ be these two sides of $P$.
 \item Let $e'$ be the vertex of $P$ that is adjacent to $c'$ but different than $a$ and, similarly, let $e''$ be the vertex of $P$ that is adjacent to $c''$ but different than $a$. Then the sides $c'e'$ and $c''e''$ of $P$ are parallel and have the same length. 
 \item The corner angles of $P$ at the vertices $e'$ and $e''$ equal $\pi/2$. 
 \end{enumerate}
\end{defn}

Using Definition \ref{defn_spec_polyg_p3}, the main result of \cite{vor-graphs} becomes the following.

\begin{thm} \label{thm_spec_polyg_p3}
If $p=3$, then the (isometric) planar polygons $P_A$ and $P_B$ are $3$-special. Furthermore, the hex sphere $M$ can be recovered from the disjoint union of the polygons $P_A$ and $P_B$ by identifying the edges on the boundaries of $P_A$ and $P_B$ as shown in Figure \ref{fig_spec_polyg_p3}. 
\end{thm}

\begin{figure}[ht!]
\labellist
\small\hair 3pt

\pinlabel $a$ [b] at 86 655
\pinlabel $c'$ [b] at -57 730
\pinlabel $c''$ [b] at 233 730
\pinlabel $e'$ [t] at -57 525
\pinlabel $e''$ [t] at 231 525
\pinlabel $\frac{4\pi}{3}$ [t] at 86 628
\pinlabel $\frac{\pi}{3}$ [tl] at -53 709
\pinlabel $\frac{\pi}{3}$ [tr] at 225 709
\pinlabel $P_A$ at 9 588

\pinlabel $b$ [b] at 520 654
\pinlabel $d'$ [b] at 379 730
\pinlabel $d''$ [b] at 668 730
\pinlabel $e'''$ [t] at 382 525
\pinlabel $e''''$ [t] at 671 525
\pinlabel $\frac{4\pi}{3}$ [t] at 520 628
\pinlabel $\frac{\pi}{3}$ [tl] at 381 709
\pinlabel $\frac{\pi}{3}$ [tr] at 659 709
\pinlabel $P_B$ at 444 588
\endlabellist
\centering
  \includegraphics[width=0.5\textwidth]{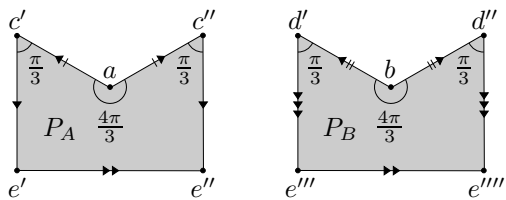}
  \caption{Recovering $M$ from $P_A \sqcup P_B$  when $p=3$}
  \label{fig_spec_polyg_p3}
\end{figure}

\begin{defn} \label{defn_spec_polyg_p4}
A planar polygon $P$ is \emph{$4$-special} if it satisfies conditions $(2)-(6)$ from Definition \ref{defn_spec_polyg_p3} as well as the following:
\begin{enumerate}
 \item $P$ has $6$ sides.
 \item Let $f'$ be the only vertex of $P$ that is different from $a$, $c'$, $c''$, $e'$ and $e''$. Then the vertices $e'$, $e''$ and $f'$ are equidistant from the vertex $a$.
\end{enumerate}
\end{defn}

Reformulating the main result of \cite{vor-graphs} in terms of Definition \ref{defn_spec_polyg_p4}, we get:

\begin{thm} \label{thm_spec_polyg_p4}
If $p=4$, then the (isometric) planar polygons $P_A$ and $P_B$ are $4$-special. Furthermore, the hex sphere $M$ can be recovered from the disjoint union of the polygons $P_A$ and $P_B$ by identifying the edges on the boundaries of $P_A$ and $P_B$ as shown in Figure \ref{fig_spec_polyg_p4}. 
\end{thm}

\begin{figure}[ht!]
\labellist
\small\hair 3pt

\pinlabel $a$ [b] at 70 687
\pinlabel $c'$ [r] at -66 768
\pinlabel $c''$ [l] at 210 768
\pinlabel $e'$ [r] at -68 533
\pinlabel $e''$ [l] at 211 533
\pinlabel $f'$ [t] at 98 478
\pinlabel $\frac{4\pi}{3}$ [t] at 70 660
\pinlabel $\frac{\pi}{3}$ [tl] at -70 740
\pinlabel $\frac{\pi}{3}$ [tr] at 208 740
\pinlabel $P_A$ at 98 550

\pinlabel $b$ [b] at 516 687
\pinlabel $d'$ [r] at 380 768
\pinlabel $d''$ [l] at 655 768
\pinlabel $f''$ [r] at 379 533
\pinlabel $f'''$ [l] at 655 533
\pinlabel $e'''$ [t] at 544 478
\pinlabel $\frac{4\pi}{3}$ [t] at 516 660
\pinlabel $\frac{\pi}{3}$ [tl] at 376 740
\pinlabel $\frac{\pi}{3}$ [tr] at 654 740
\pinlabel $P_B$ at 544 550
\endlabellist
\centering
  \includegraphics[width=0.5\textwidth]{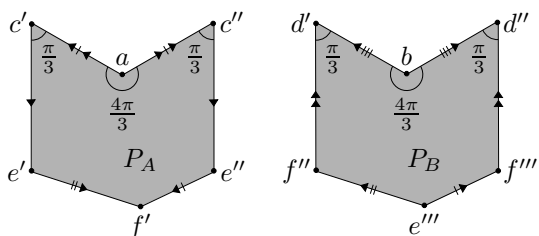}
  \caption{Recovering $M$ from $P_A \sqcup P_B$  when $p=4$}
  \label{fig_spec_polyg_p4}
\end{figure}

\begin{defn} \label{defn_spec_polyg}
A planar polygon $P$ is \emph{special} if it is $n$-special for some $n \in \{2,3,4\}$.
\end{defn}

\begin{rmk} \label{rmk_gen_spec_polyg}
Observe that a ``generic'' special polygon is $4$-special. In fact:
\begin{itemize}
 \item A $3$-special polygon is a degenerate $4$-special polygon in which a vertex merges with one of its two adjacent vertices.
 \item A $2$-special polygon is a degenerate $4$-special polygon whose pair of parallel sides  collapse to points.
\end{itemize}
\end{rmk}

\begin{defn} \label{defn_angle-param}
 Let $P$ be a special polygon and let $a$ be the unique vertex of $P$ at which the corner angle equals $4\pi/3$. Suppose that if we traverse the boundary of $P$ counterclockwisely, starting and ending at the vertex $a$, we encounter the vertices 
\begin{itemize}
 \item $a,c',d,c'',a$ if $P$ is $2$-special;
 \item $a,c',e',e'',c'',a$ if $P$ is $3$-special;
 \item $a,c',e',f',e'',c'',a$ if $P$ is $4$-special.
\end{itemize}
We define the \emph{angles} $\varphi=\varphi(P)$, $\alpha=\alpha(P)$ and $\beta=\beta(P)$ as follows:
\begin{itemize}
 \item $\varphi=0$, $\alpha=\angle ac'd$ and $\beta=\angle ac''d$ if $P$ is $2$-special;
 \item $\varphi=\angle c'ae'$, $\alpha=\angle ae'e''$ and $\beta=\angle ae''e'$ if $P$ is $3$-special;
 \item $\varphi=\angle c'ae'$, $\alpha=\angle ae'f'$ and $\beta=\angle ae''f'$ if $P$ is $4$-special.
\end{itemize}
\end{defn}

\begin{figure}[ht!]
\labellist
\small\hair 2pt

\pinlabel $a$ [b] at 82 760
\pinlabel $c'$ [b] at -29 818
\pinlabel $c''$ [b] at 194 818
\pinlabel $d$ [t] at 75 619
\pinlabel $\alpha$ [r] at 17 787
\pinlabel $\beta$ [r] at 175 782
\pinlabel {$2$-special} at 82.5 820

\pinlabel $a$ [b] at 479 741
\pinlabel $c'$ [b] at 364 798
\pinlabel $e'$ [t] at 364 637
\pinlabel $e''$ [t] at 592 637
\pinlabel $c''$ [b] at 592 798
\pinlabel $\alpha$ [l] at 384 649
\pinlabel $\beta$ [r] at 562 653
\pinlabel $\varphi$ [r] at 459 733
\pinlabel {$3$-special} at 478 800

\pinlabel $a$ [b] at 282 455
\pinlabel $c'$ [b] at 171 512
\pinlabel $e'$ [t] at 171 332
\pinlabel $f'$ [t] at 303 290
\pinlabel $e''$ [t] at 392 334
\pinlabel $c''$ [b] at 391 512
\pinlabel $\alpha$ [l] at 190 339
\pinlabel $\beta$ [r] at 373 334
\pinlabel $\varphi$ [r] at 262 444
\pinlabel {$4$-special} at 281 514

\pinlabel $\frac{\pi}{3}$ [tr] at 585 781
\pinlabel $\frac{\pi}{3}$ [tl] at 372 781

\pinlabel $\frac{\pi}{3}$ [tl] at 173 495
\pinlabel $\frac{\pi}{3}$ [tr] at 388 496

\endlabellist
\centering
  \includegraphics[width=0.55\textwidth]{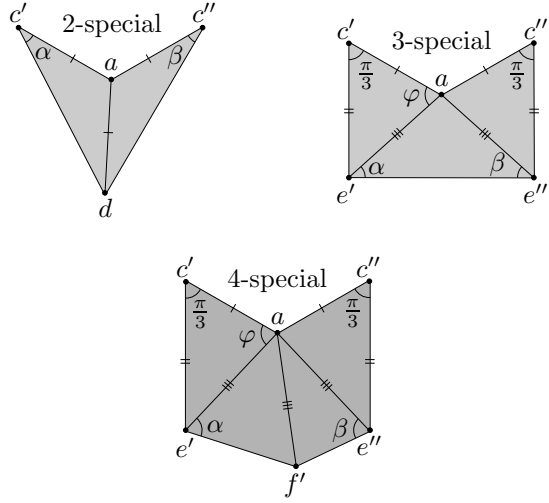}
  \caption{Defining the angles $\varphi$, $\alpha$ and $\beta$}
  \label{fig_defn_phi_alpha_beta}
\end{figure}

Given a special polygon $P$, it is easy to see that $\alpha+\beta=\varphi+\pi/3$ if $P$ is either $2$- or $4$-special, and that $\alpha=\beta=\varphi-\pi/6$ if $P$ is $3$-special.

\begin{defn} \label{ang-param-hex-sph}
Given a hex sphere $M$ with Voronoi  polygons $P_A$ or $P_B$, we define its angles parameters $\varphi(M)$, $\alpha(M)$ and $\beta(M)$ by $\varphi(M)=\varphi(P_A)=\varphi(P_B)$, $\alpha(M)=\alpha(P_A)=\alpha(P_B)$ and $\beta(M)=\beta(P_A)=\beta(P_B)$.
\end{defn}

\section{The Moduli Space of Special Polygons} \label{sect-mod-space-spec-polyg}

\begin{defn} \label{def-MSP}
We define the equivalence relation $\sim$ on the set of all special polygons by  $P_1 \sim P_2$ if and only if there is an orientation-preserving similarity of the plane which sends the vertices of $P_1$ to the vertices of $P_2$. The set of all $\sim$ equivalence classes is called the \emph{moduli space} of special polygons and it is denoted by $\mathcal{MSP}$. The set $\mathcal{MSP}_n \subset \mathcal{MSP}$ is the set of all $\sim$ equivalence classes of $n$-special polygons ($n \in \{2,3,4\}$).
\end{defn}

\begin{notn} From now on, we will use the following notation:
\begin{itemize}
\item The $\sim$ equivalence class of a special polygon $P$ will be denoted by $[P]_{\text{sim}} \in \mathcal{MSP}$. 
\item The only vertex of a special polygon at which the corner angle is $4\pi/3$ will be denoted by $a$.
\item Given a special polygon $P$, let $\widehat{P}$ denote the only \emph{unit-area} polygon that is the image of the polygon $P$ under a homothety of the plane fixing the vertex $a$. 
\item If $P$ is a special polygon, then we write $P=(v_1,v_2, \cdots,v_n)$ to mean that $v_1,v_2, \cdots,v_n$ (in that order) are the vertices of $P$ that we find when we travel along the boundary of $P$ counterclockwisely, starting at the vertex $v_1=a$. 
\end{itemize}
\end{notn}

The proofs of Propositions \ref{prop_ang_par_cond_4-spec}-\ref{prop_ang_par_cond_2-spec} use only elementary Euclidean geometry and therefore they are omitted.

\begin{prop} \label{prop_ang_par_cond_4-spec}
Given a $4$-special polygon $P$, its angle parameters $\varphi$ and $\alpha$ satisfy the following inequalities:
\begin{enumerate}
 \item $0 < \varphi < 2\pi/3$,
 \item $0 < \alpha < \pi/2$,
 \item $\varphi-\pi/6 < \alpha < \varphi+\pi/3$.
\end{enumerate}
Conversely, given two numbers $\varphi$ and $\alpha$ satisfying (1)-(3), there is a $4$-special polygon $P$ with  $\varphi=\varphi(P)$ and $\alpha=\alpha(P)$. If $P_1$ is another $4$-special polygon with $\varphi=\varphi(P_1)$ and $\alpha=\alpha(P_1)$, then $P \sim P_1$.
\end{prop}

\begin{prop} \label{prop_ang_par_cond_3-spec}
Given a $3$-special polygon $P$, its angle parameters $\varphi$ and $\alpha$ satisfy the following inequalities:
\begin{enumerate}
 \item $\pi/6 < \varphi < 2\pi/3$,
 \item $\alpha=\varphi-\pi/6$.
 \end{enumerate}
Conversely, given two numbers $\varphi$ and $\alpha$ satisfying (1)-(2), there is a $3$-special polygon $P$ with  $\varphi=\varphi(P)$ and $\alpha=\alpha(P)$. If $P_1$ is another $3$-special polygon with $\varphi=\varphi(P_1)$ and $\alpha=\alpha(P_1)$, then $P \sim P_1$.
\end{prop}

\begin{prop} \label{prop_ang_par_cond_2-spec}
Given a $2$-special polygon $P$, its angle parameters $\varphi$ and $\alpha$ satisfy the following inequalities:
\begin{enumerate}
 \item $\varphi=0$,
 \item $0<\alpha<\pi/3$.
 \end{enumerate}
Conversely, given two numbers $\varphi$ and $\alpha$ satisfying (1)-(2), there is a $2$-special polygon $P$ with  $\varphi=\varphi(P)$ and $\alpha=\alpha(P)$. If $P_1$ is another $2$-special polygon with $\varphi=\varphi(P_1)$ and $\alpha=\alpha(P_1)$, then $P \sim P_1$.
\end{prop}

We know from Section \ref{sec-ang-param-spec-polyg} that given a special polygon $P$ we can construct a hex sphere in the following way. Take the polygon $P$ and a copy, $P_1$, of P. Identify the edges on the boundaries of $P$ and $P_1$ according to the gluing scheme of Figure \ref{fig_spec_polyg_p2}, Figure \ref{fig_spec_polyg_p3} or Figure \ref{fig_spec_polyg_p4} (depending if $P$ is $2$-, $3$- or $4$-special, respectively). Let $Hex(P)$ denote the hex sphere constructed from $P$ in this way, and let $\pi_P \colon P \sqcup P_1 \to Hex(P)$ be the canonical map. Let $a_P \in Hex(P)$ be the image of the vertex $a$ of $P$. Let $A_P$ be the Voronoi cell of the hex sphere $Hex(P)$ centered at the singular point $a_P$. Then the polygon $P$ can be identified with the planar polygon associated to $A_P$.

\begin{defn}
 Let $P$ and $Q$ be two special polygons. A PL homeomorphism $g \colon P \to Q$ is said to \emph{extend} to a map $\bar{g} \colon Hex(P) \to Hex(Q)$ if  $\bar{g} \circ \pi_P \circ i_P=\pi_Q \circ i_Q \circ g$, where $i_P \colon P \hookrightarrow P \sqcup P_1$ and $i_Q \colon Q \hookrightarrow Q \sqcup Q_1$ are the inclusion maps.
\end{defn}

\begin{defn} \label{def_met_MSP}
Given two special polygons $P$ and $Q$, we define the number
$d_{\mathcal{SP}}(P,Q)=\log \inf K$. The infimum is taken over all numbers $K \geq 1$ such that there is an orientation-preserving PL $K$-bi-Lipschitz homeomorphism $g \colon \widehat{P} \to \widehat{Q}$, which  extends to an orientation-preserving PL homeomorphism  $\bar{g} \colon Hex(\widehat{P}) \to Hex(\widehat{Q})$ that preserves Voronoi graphs.
\end{defn}

\begin{prop} \label{prop_met_MSP}
The map $\textbf{d}_{\mathcal{MSP}}$ of $\mathcal{MSP} \times \mathcal{MSP}$ into the non-negative reals defined by
$\textbf{d}_{\mathcal{MSP}}([P]_{\text{sim}},[Q]_{\text{sim}})=d_{\mathcal{SP}}(P,Q)$
is a metric on $\mathcal{MSP}$. Let $Cl(\cdot)$ denote the closure of $\cdot$ in the metric space $(\mathcal{MSP},\textbf{d}_{\mathcal{MSP}})$. Then $Cl(\mathcal{MSP}_4)=\mathcal{MSP}$, and $\mathcal{MSP}_i \cap Cl(\mathcal{MSP}_j)=\emptyset$  for $i,j \in \{2,3\}$, $i \neq j$.
\end{prop}
The proof of the Proposition \ref{prop_met_MSP} follows from standard metric arguments.

\begin{defn} \label{def-set-Y} Let $\mathcal{Y}=\{(\varphi,\alpha) \in [0,2\pi/3) \times (0,\pi/2] \colon \varphi-\pi/6 \leq \alpha < \varphi+\pi/3\}$. This set becomes a metric space with the restriction of the Euclidean metric on the plane to $\mathcal{Y}$. 
\end{defn}
The reader can check that the map $\Lambda \colon \mathcal{MSP} \to \mathcal{Y}$ defined by \begin{equation} \label{defn-Lambda}
 \Lambda([P]_{\text{sim}})=(\varphi(P),\alpha(P)), \, [P]_{\text{sim}} \in \mathcal{MSP}.
\end{equation}
is well-defined. However, this map is not continuous (see Example \ref{ex-prop-justif-defn-Z}).

\begin{examp} \label{ex-prop-justif-defn-Z} 
Consider the unit-area, $4$-special polygons $P_1$ and $P_2$ from Figure \ref{fig-justif-defn-Z_1}. Suppose that $d_{\mathbb{E}^2}(a_1,c_1')=d_{\mathbb{E}^2}(a_2,c_2')$ and that $d_{\mathbb{E}^2}(c_1',e_1')=d_{\mathbb{E}^2}(c_2',e_2')$ (where $d_{\mathbb{E}^2}$ denotes the Euclidean distance on the plane). Let $PENT=(a_1,c_1',e_1',e_1'',c_1'')$ (see Figure \ref{fig-justif-defn-Z_1}). Then clearly:
\begin{enumerate}
\item  $\lim_{f_1 \to e_1'}([P_1]_{\text{sim}})=[PENT]_{\text{sim}}$=$\lim_{f_2 \to e_2''}([P_2]_{\text{sim}})$. [Here $f_1 \to e_1'$ (respectively, $f_2 \to e_2''$) means that $f_1$ (respectively, $f_2$) approaches $e_1'$(respectively, $e_2''$) keeping all vertices of $P_1$(respectively, $P_2$) but $f_1$(respectively, $f_2$) fixed.]
\item $\lim_{f_1 \to e_1'}(\Lambda(P_1))=(\varphi_1,\pi/2)$, but $\lim_{f_2 \to e_2''}(\Lambda(P_2))=(\varphi_1,\varphi_1-\pi/6)$.
\end{enumerate}

\vspace{12pt} 
\begin{figure}[ht!]
\labellist
\small\hair 2pt

\pinlabel $\frac{\pi}{3}$ [tl] at 5 693
\pinlabel $\frac{\pi}{3}$ [tr] at 223 693

\pinlabel $\frac{\pi}{3}$ [tl] at 374 693
\pinlabel $\frac{\pi}{3}$ [tr] at 592 693

\pinlabel $a_1$ [b] at 114 650
\pinlabel $c_1'$ [b] at 3 708
\pinlabel $c_1''$ [b] at 224 708
\pinlabel $e_1'$ [r] at 1 527 
\pinlabel $e_1''$ [l] at 226 527
\pinlabel $f_1$ [t] at 19 514

\pinlabel $a_2$ [b] at 483 650
\pinlabel $c_2'$ [b] at 372 708
\pinlabel $c_2''$ [b] at 593 708
\pinlabel $e_2'$ [r] at 372 527 
\pinlabel $e_2''$ [l] at 595 527
\pinlabel $f_2$ [t] at 577 514

\pinlabel $\varphi_1$ [r] at 102 641
\pinlabel $\alpha_1$ [l] at 18 532
\pinlabel $\varphi_2$ [r] at 471 641
\pinlabel $\alpha_2$ [l] at 388 537

\pinlabel $P_1$ at 114 575 
\pinlabel $P_2$ at 483 575 
\endlabellist
\centering
  \includegraphics[width=0.5\textwidth]{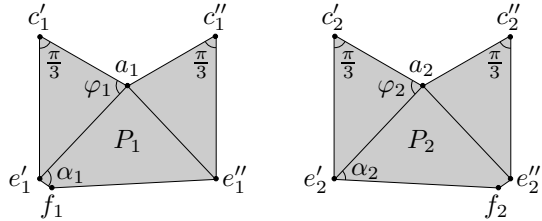}
  \caption{The special polygons $P_1$ and $P_2$}
  \label{fig-justif-defn-Z_1}
\end{figure}
\end{examp}
Example \ref{ex-prop-justif-defn-Z} motivates the following definition.
\begin{defn} \label{def-set-Z}
Let $\mathcal{Z}$ be the set obtained from $\mathcal{Y}$ by identifying the points $(\varphi,\pi/2) \in \mathcal{Y}$ and $(\varphi,\varphi-\pi/6) \in \mathcal{Y}$ for $\pi/6<\varphi<2\pi/3$. The set $\mathcal{Z}$ inherits a metric from $\mathcal{Y}$ and, with this metric, $\mathcal{Z}$ becomes a metric surface. The canonical projection of $ \mathcal{Y}$ onto $ \mathcal{Z}$ will be denoted by $\Pi \colon \mathcal{Y} \to \mathcal{Z}$.
\end{defn}

\begin{thm} \label{thm_mod_spac_polyg} 
Let $\Lambda \colon \mathcal{MSP} \to \mathcal{Y}$ be the map defined in \eqref{defn-Lambda}. Then the composition  $\Pi \circ \Lambda \colon \mathcal{MSP} \to \mathcal{Z}$  is a homeomorphism. 
\end{thm}

\begin{proof} The proof is divided  into three steps. 
 
\textbf{$\diamond$ Step I:} \emph{The map $\Pi \circ \Lambda \colon \mathcal{MSP} \to \mathcal{Z}$ is  bijective}.

Let $\widetilde{\mathcal{Y}}=\{(\varphi,\alpha) \in \mathcal{Y} \colon \alpha \neq \pi/2\}$.  By Propositions \ref{prop_ang_par_cond_4-spec}, \ref{prop_ang_par_cond_3-spec} and \ref{prop_ang_par_cond_2-spec}, $\widetilde{\mathcal{Y}}=\Lambda(\mathcal{MSP})$. Clearly, the restriction of the map $\Pi$ to $\widetilde{\mathcal{Y}}$ is a bijection of $\widetilde{\mathcal{Y}}$ onto $\mathcal{Z}$. Thus, $\Pi \circ \Lambda(\mathcal{MSP})=\Pi(\widetilde{\mathcal{Y}})=\mathcal{Z}$. Since the restriction of $\Pi$ to $\widetilde{\mathcal{Y}}=\Lambda(\mathcal{MSP})$ is injective, then it suffices to show that $\Lambda \colon \mathcal{MSP} \to \mathcal{Y}$ is injective. This can be easily done by using Propositions \ref{prop_ang_par_cond_4-spec}, \ref{prop_ang_par_cond_3-spec} and \ref{prop_ang_par_cond_2-spec}.   

\textbf{$\diamond$ Step II:} \emph{The map $\Pi \circ \Lambda \colon \mathcal{MSP} \to \mathcal{Z}$ is continuous}.

For each $n$, let $\Lambda_n \colon \mathcal{MSP}_n \to \mathcal{Y}$ be the restriction of $\Lambda$ to $\mathcal{MSP}_n$. Since the map $\Lambda_n$ is continuous, then  so is  $\Pi \circ \Lambda_n \colon \mathcal{MSP}_n \to \mathcal{Z}$. Hence, $\Pi \circ \Lambda \colon \mathcal{MSP} \to \mathcal{Z}$ is also continuous (by Remark \ref{rmk_gen_spec_polyg} and Proposition \ref{prop_met_MSP}).

\textbf{$\diamond$ Step III:} \emph{The inverse map of  $\Pi \circ \Lambda \colon \mathcal{MSP} \to \mathcal{Z}$ is also continuous}.

Consider the map $\Phi \colon \mathcal{Y} \to \mathcal{MSP}$ defined by $\Phi(\varphi,\alpha)=[P]_{\text{sim}}$, where $P$ is a special polygon with angle parameters $(\varphi,\alpha)$. The reader can check that this map is well-defined. The map $\Phi$ is continuous because two (unit-area) special polygons are close in the bi-Lipschitz distance provided that their angle parameters are sufficiently close in the space $\mathcal{Y}$. Also, by Example \ref{ex-prop-justif-defn-Z}, $\Phi(\varphi,\pi/2)=\Phi(\varphi,\varphi-\pi/6)$, and so $\Phi \colon \mathcal{Y} \to \mathcal{MSP}$ descends to a continuous map $\widetilde{\Phi} \colon \mathcal{Z} \to \mathcal{MSP}$ such that $\widetilde{\Phi} \circ \Pi=\Phi$. Clearly, $\widetilde{\Phi} \colon \mathcal{Z} \to \mathcal{MSP}$ is the inverse map of  $\Pi \circ \Lambda \colon \mathcal{MSP} \to \mathcal{Z}$.
\end{proof}

\begin{rmk} \label{rmk-decomp-polyh-Z}
By Theorem \ref{thm_mod_spac_polyg}, we can decompose the surface $\mathcal{Z}$ as   $\mathcal{Z}= \mathcal{Z}_2 \sqcup \mathcal{Z}_3 \sqcup \mathcal{Z}_4$, where $\mathcal{Z}_2=\Pi(\{(\varphi,\alpha) \in \mathcal{Y} \colon \varphi=0\})$, $\mathcal{Z}_3=\Pi(\{(\varphi,\alpha) \in \mathcal{Y} \colon \alpha=\varphi-\pi/6\}$ and  $\mathcal{Z}_4=\Pi(\{(\varphi,\alpha) \in \mathcal{Y} \colon \varphi-\pi/6 < \alpha < \varphi+\pi/3\})$. Notice that $\mathcal{Z}_n$ parametrizes $\mathcal{MSP}_n$, i.e., $\mathcal{Z}_n=\Pi \circ \Lambda(\mathcal{MSP}_n)$ for each $n=2,3,4$. 
\end{rmk}

\section{The Moduli Space of Hex Structures of Unit Area}

\begin{defn} \label{def-mod-sp-hex-str-unit-area}
 Given a topological hex surface $S$, let $\mathcal{M}_{\text{unit}}(S)$ denote the set of all hex structures marked by $S$ that have unit area.
\end{defn}
The set $\mathcal{M}_{\text{unit}}(S)$ is a subspace of $\mathcal{M}(S)$ with the restriction of the metric $\textbf{d}_{\mathcal{M}(S)}$ to $\mathcal{M}_{\text{unit}}(S)$. 

\begin{defn} \label{defn_perm_ind_con_pts}
Let $S$ be a topological hex sphere  with cone points $x_1$, $x_2$, $x_3$ and $x_4$. From now on, we will use the following notation:
\begin{itemize}
\item $\{i_a,i_b,i_c,i_d\}$ will be the unique permutation of the set $\{1,2,3,4\}$ that satisfies the following properties:
\begin{itemize}
 \item $x_{i_a}$ and $x_{i_b}$ are the cone points of $S$ of angle $4\pi/3$.
 \item $x_{i_c}$ and $x_{i_d}$ are the cone points of $S$ of angle $2\pi/3$.
 \item $i_a \leq i_b$ and $i_c \leq i_d$.
\end{itemize}
\item Given a hex structure $(M,h)$ marked by $S$, we define $a=h(x_{i_a})$, $b=h(x_{i_b})$, $c=h(x_{i_c})$ and $d=h(x_{i_d})$.
\end{itemize}

\end{defn}

The definition below gives us a way to split $\mathcal{M}_{\text{unit}}(S)$ into ``halves''.

\begin{defn} \label{def-split-mod-sp-hex-str-unit-area}
 Given a topological hex surface $S$, define
\begin{itemize}
 \item $\mathcal{M}^{+}_{\text{unit}}(S)=\{[M,h]_{\simeq_\mathcal{M}} \in \mathcal{M}_{\text{unit}}(S) \colon d_M(a,c) \leq d_M(b,c)\}$;
 \item $\mathcal{M}^{-}_{\text{unit}}(S)=\{[M,h]_{\simeq_\mathcal{M}} \in \mathcal{M}_{\text{unit}}(S) \colon d_M(b,c) \leq d_M(a,c)\}$.
\end{itemize}
\end{defn}

The spaces $\mathcal{M}^{+}_{\text{unit}}(S)$ and $\mathcal{M}^{-}_{\text{unit}}(S)$ have some nice properties.
\begin{prop} \label{prop-split-mod-sp-hex-str-unit-area}  Given a topological hex surface $S$, the following hold:
\begin{enumerate}
 \item The sets $\mathcal{M}^{+}_{\text{unit}}(S)$ and $\mathcal{M}^{-}_{\text{unit}}(S)$ are closed in $\mathcal{M}_{\text{unit}}(S)$.
 \item The space $\mathcal{M}_{\text{unit}}(S)$ is the union of $\mathcal{M}^{+}_{\text{unit}}(S)$ and $\mathcal{M}^{-}_{\text{unit}}(S)$.
 \item A point $[M,h]_{\simeq_\mathcal{M}} \in \mathcal{M}_{\text{unit}}(S)$ is in $\mathcal{M}^{+}_{\text{unit}}(S) \cap \mathcal{M}^{-}_{\text{unit}}(S)$ if and only if both Voronoi cells $A$ and $B$ of $M$ have exactly $2$ edges.
\end{enumerate}
\end{prop} 

\begin{proof}
The statements (1) and (2) are obvious. We now prove (3). Suppose that $[M,h]_{\simeq_\mathcal{M}} \in \mathcal{M}^{+}_{\text{unit}}(S) \cap \mathcal{M}^{-}_{\text{unit}}(S)$, so that $d_M(a,c)=d_M(b,c)$. This means that there are at least two shortest geodesics from $c$ to the set $\{a,b\}$. Hence, the degree of the vertex $c$ of the graph $\Gamma$, the Voronoi graph of $M$, is at least $2$ (by Lemma 5.3 
from \cite{vor-graphs}). Similarly, the degree of the vertex $d$ of the graph $\Gamma$ is also at least $2$, since  $d_M(a,c)=d_M(b,c)$ implies that  $d_M(b,d)=d_M(a,d)$ (by Theorem 4.4 
from \cite{vor-graphs}).

Let $p$ be the number of edges of the Voronoi cell $A$ of the hex sphere $M$. By Theorem 1.2 and Observation 6.1 from \cite{vor-graphs}, $p$ is also the number of edges of the Voronoi cell $B$ and it can only be equal to $2$, $3$ or $4$. If $p$ were equal to either $3$ or $4$, then the vertices $c$ and $d$ of the Voronoi graph $\Gamma$ would have degree $1$ (by Lemmas 6.4 and 6.5 from \cite{vor-graphs}). But this contradicts what we just proved in the previous paragraph. Therefore, $p=2$.

Conversely, suppose that $[M,h]_{\simeq_\mathcal{M}} \in \mathcal{M}_{\text{unit}}(S)$ and that 
both Voronoi cells $A$ and $B$ of $M$ have exactly $2$ edges. This implies that $d_M(a,c)=d_M(b,c)$ (by the proof of Lemma 6.2 from \cite{vor-graphs}). Therefore, $[M,h]_{\simeq_\mathcal{M}} \in \mathcal{M}^{+}_{\text{unit}}(S) \cap \mathcal{M}^{-}_{\text{unit}}(S)$.
\end{proof}

\begin{defn} \label{def-set-Z-hat}
Consider two copies $\mathcal{Z}^{+}$ and $\mathcal{Z}^{-}$ of the surface $\mathcal{Z}$ from Definition \ref{def-set-Z}. By Remark \ref{rmk-decomp-polyh-Z},  $\mathcal{Z}^{+}=\mathcal{Z}_2^{+} \sqcup \mathcal{Z}_3^{+} \sqcup \mathcal{Z}_4^{+}$ and $\mathcal{Z}^{-}=\mathcal{Z}_2^{-} \sqcup \mathcal{Z}_3^{-} \sqcup \mathcal{Z}_4^{-}$, where $\mathcal{Z}_n^{+}$ (respectively, $\mathcal{Z}_n^{-}$) is the subspace of $\mathcal{Z}^{+}$ (respectively, $\mathcal{Z}^{-}$) parametrizing  $\mathcal{MSP}_n$. We define  $\widehat{\mathcal{Z}}$ as the space obtained from $\mathcal{Z}^{+} \sqcup \mathcal{Z}^{-}$ by identifying a point in $\mathcal{Z}_2^{+}$ with its corresponding point in $\mathcal{Z}_2^{-}$.
\end{defn}

The surface $\widehat{\mathcal{Z}}$ is homeomorphic to a sphere with three holes, without including the boundary of the holes  (see Figure \ref{pts-in-Z-hat-as-HS}). 

\begin{thm} \label{thm-hex-spher-to-spec-polyg}
Define a map  $\Delta^{+} \colon \mathcal{M}^{+}_{\text{unit}}(S) \to \mathcal{MSP}$ by 	$\Delta^{+}([M,h]_{\simeq_\mathcal{M}})=[P_A]_{\text{sim}}$, where $P_A$ is the special polygon associated to the Voronoi region $A$ of $M$. Then $\Delta^{+}$ is a homeomorphism. Similarly, the map  $\Delta^{-} \colon \mathcal{M}^{-}_{\text{unit}}(S) \to \mathcal{MSP}$ defined by $\Delta^{-}([M,h]_{\simeq_\mathcal{M}})=[P_A]_{\text{sim}}$ is also a homeomorphism.
\end{thm}

\begin{proof} The reader can check that the map  $\Delta^{+}$ is well-defined. We now divide the proof into three steps.

\textbf{$\diamond$ Step I:} \emph{The map  $\Delta^{+} \colon \mathcal{M}^{+}_{\text{unit}}(S) \to \mathcal{MSP}$ is continuous}.

Let $\{[M_n,h_n]_{\simeq_\mathcal{M}}\}_n$ be a sequence of points in $\mathcal{M}^{+}_{\text{unit}}(S)$ converging to a point $[M',h']_{\simeq_\mathcal{M}} \in 
\mathcal{M}^{+}_{\text{unit}}(S)$. This means that there is a sequence $\{K_n\}_n$ converging to $1$, and for each $n$ there is an orientation-preserving PL $K_n$-bi-Lipschitz homeomorphism $g_n \colon M_n \to M'$ such that $g_n \circ h_n|_{\Sigma}=h'|_{\Sigma}$ and $g_n(\Gamma_n)=\Gamma'$. [Here $\Sigma$ is the singular locus of the topological hex sphere $S$, and $\Gamma_n$, $\Gamma'$ are the Voronoi graphs of $M_n$, $M'$, respectively.] 

Let $A_n$ be the Voronoi cell of $M_n$ centered at $a_n$, and let $A'$ be the Voronoi cell of $M'$ centered at $a'$. Standard topological arguments show that $g_n(A_n)=A'$. Thus, $g_n$ induces an orientation-preserving PL $K_n$-bi-Lipschitz homeomorphism $\tilde{g}_n \colon P_{A_n} \to P_{A'}$, where $P_{A_n}$ and $P_{A'}$ are the special polygons associated to the Voronoi cells $A_n$ and $A'$, respectively. The map $\tilde{g}_n \colon P_{A_n} \to P_{A'}$ naturally gives rise to a $K_n$-bi-Lipschitz homeomorphism  $\widehat{g}_n \colon \widehat{P}_{A_n} \to \widehat{P}_{A'}$, which extends to an orientation-preserving PL homeomorphism  $\bar{g}_n \colon Hex(\widehat{P}_{A_n}) \to Hex(\widehat{P}_{A'})$. Further, $\bar{g}_n$ preserves Voronoi graphs.  Thus,  $\Delta^{+}([M_n,h_n]_{\simeq_\mathcal{M}}) \to \Delta^{+}([M',h']_{\simeq_\mathcal{M}})$ as $n \to \infty$.

\textbf{$\diamond$ Step II:} \emph{The map  $\Delta^{+} \colon \mathcal{M}^{+}_{\text{unit}}(S) \to \mathcal{MSP}$ is bijective}.

Let $[P]_{\text{sim}} \in \mathcal{MSP}$. Pick a polygon $P_{1/2}$ in $[P]_{\text{sim}}$ with area $(P_{1/2})=1/2$. Then the hex sphere $M=Hex(P_{1/2})$ has unit area. Pick a marking $h \colon S \to M$, so that $[M,h]_{\simeq_\mathcal{M}} \in \mathcal{M}^{+}_{\text{unit}}(S)$ and $\Delta^{+}([M,h]_{\simeq_\mathcal{M}})= [P]_{\text{sim}}$.

Suppose that $[M,h]_{\simeq_\mathcal{M}}$ and $[M',h']_{\simeq_\mathcal{M}}$ are two points in $\mathcal{M}^{+}_{\text{unit}}(S)$ with $\Delta^{+}([M,h]_{\simeq_\mathcal{M}})=\Delta^{+}([M',h']_{\simeq_\mathcal{M}})$. Then $[P_A]_{\text{sim}}=[P_{A'}]_{\text{sim}}$, where $P_A$ (respectively, $P_{A'}$) is the planar polygon associated to the Voronoi region $A$ (respectively, $A'$) of $M$ (respectively, $M'$). Thus, there is an orientation-preserving similarity of the plane sending $P_A$ to $P_{A'}$.  This similarity is actually an isometry, as the area of $P_A$ equals that of $P_{A'}$. This isometry extends to an isometry of $M=Hex(P_A)$ to $M'=Hex(P_{A'})$ that preserves the cone points and the Voronoi graphs. Hence, $[M,h]_{\simeq_\mathcal{M}}=[M',h']_{\simeq_\mathcal{M}}$.

\textbf{$\diamond$ Step III:} \emph{The inverse map of $\Delta^{+}$ is also continuous}.

This step is proven using the arguments from Step I.
\end{proof}

We now prove that the surface $\widehat{\mathcal{Z}}$  is a parameter space for $\mathcal{M}_{\text{unit}}(S)$.

\begin{thm} \label{thm-mod-spac-hex-spher}
Let $\mathcal{Z}^{+}$ and $\mathcal{Z}^{-}$ be the two copies of $\mathcal{Z}$ from Definition \ref{def-set-Z-hat}. Let $id^{+} \colon \mathcal{Z} \to \mathcal{Z}^{+}$ and $id^{-} \colon \mathcal{Z} \to \mathcal{Z}^{-}$ be the identity maps. Then:

\begin{enumerate}[(i)]
 \item $Par^{+}=id^{+} \circ \Pi \circ \Lambda \circ \Delta^{+}\colon \mathcal{M}^{+}_{\text{unit}}(S) \to \mathcal{Z}^{+}$ is a homeomorphism;
 \item $Par^{-}=id^{-} \circ \Pi \circ \Lambda \circ \Delta^{-}\colon \mathcal{M}^{-}_{\text{unit}}(S) \to \mathcal{Z}^{-}$ is a homeomorphism.
\end{enumerate}
Furthermore, the maps $Par^{+}$ and $Par^{-}$ can be amalgamated to give a homeomorphism $Par \colon \mathcal{M}_{\text{unit}}(S) \to \widehat{\mathcal{Z}}$ such that $Par|_{\mathcal{M}^{+}_{\text{unit}}(S)}=Par^{+}$ and $Par|_{\mathcal{M}^{-}_{\text{unit}}(S)}=Par^{-}$.
\end{thm}

\begin{proof}
(i) and (ii) follow from theorems \ref{thm_mod_spac_polyg} and \ref{thm-hex-spher-to-spec-polyg}. By Proposition \ref{prop-split-mod-sp-hex-str-unit-area}, the space $\mathcal{M}_{\text{unit}}(S)$ is the union of $\mathcal{M}^{+}_{\text{unit}}(S)$ and $\mathcal{M}^{-}_{\text{unit}}(S)$, which intersect along the space of all $[M,h]_{\simeq_\mathcal{M}} \in \mathcal{M}_{\text{unit}}(S)$ for which both Voronoi cells $A$ and $B$ of $M$ have exactly $2$ edges. Also, $Par^{+}([M,h]_{\simeq_\mathcal{M}})=Par^{-}([M,h]_{\simeq_\mathcal{M}})$ for all $[M,h]_{\simeq_\mathcal{M}} \in \mathcal{M}^{+}_{\text{unit}}(S) \cap \mathcal{M}^{-}_{\text{unit}}(S)$. Therefore, the maps $Par^{+}$ and $Par^{-}$ can be combined to produce a homeomorphism $Par \colon \mathcal{M}_{\text{unit}}(S) \to \widehat{\mathcal{Z}}$ that restricts to $Par^{+}$ on $\mathcal{M}^{+}_{\text{unit}}$ and to  $Par^{-}$ on $\mathcal{M}^{-}_{\text{unit}}$.
\end{proof}

\begin{cor} \label{main-cor}
Let $M$ be a hex sphere of unit area. Then
\begin{enumerate}
  \item $M$ embeds isometrically in the 3-dimensional Euclidean space as the boundary of a tetrahedron.
 \item There exists an open Euclidean annulus, embedded in $M$, which is disjoint from the cone points in $M$. In particular, $M$ has a simple closed geodesic.
 \item There is a simple closed geodesic $\gamma$ in $M$ such that a fractional Dehn twist along $\gamma$ converts $M$ to the double of a perfect parallelogram.
\end{enumerate}
\end{cor}

\begin{proof}
Let $\varphi$ and $\alpha$ be the angle parameters of $M$. Let $P_{A}$ and $P_{B}$ be the Voronoi polygons of $M$. Then both $P_{A}$ and $P_{B}$ are either $2$-, $3$- or $4$-special. 

In \cite{vor-graphs} the author proves that the polygons $P_{A}$ and $P_{B}$ are isometric, and that the hex sphere $M$ can be reconstructed from $P_{A}$ and $P_{B}$ by gluing pairs of edges of these polygons according to one of $3$ possible combinatorial patterns. These gluing patterns depend on whether the Voronoi polygons are either $2$-, $3$- or $4$-special. We divide the proof into three cases.

\textbf{$\diamond$ Case I:} Both $P_{A}$ and $P_{B}$ are $2$-special.

The gluing pattern to obtain $M$ from $P_{A}$ and $P_{B}$ is that of Figure \ref{fig_spec_polyg_p2}. Let $Z$ be the planar polygon obtained from $P_A \sqcup P_B$ by identifying the edges labeled by {\tiny $\blacktriangleright\blacktriangleright$} in Figure \ref{fig_spec_polyg_p2}. Figure \ref{cor-p2} shows the polygon $Z$ and the identification pattern on its boundary needed to recover $M$.

\begin{figure}[ht!]
\labellist
\small\hair 2pt
\pinlabel $c'$ [b] at -119 622
\pinlabel $c''$ [b] at 145 622
\pinlabel $a$ [b] at 9 554
\pinlabel $d$ [t] at 17 388
\pinlabel $c'''$ [l] at 287 414
\pinlabel $b$ [l] at 155 489
\vspace*{11pt}
\endlabellist
\subfloat[The polygon $Z$]{\label{cor-p2}\includegraphics[width=0.2\textheight]{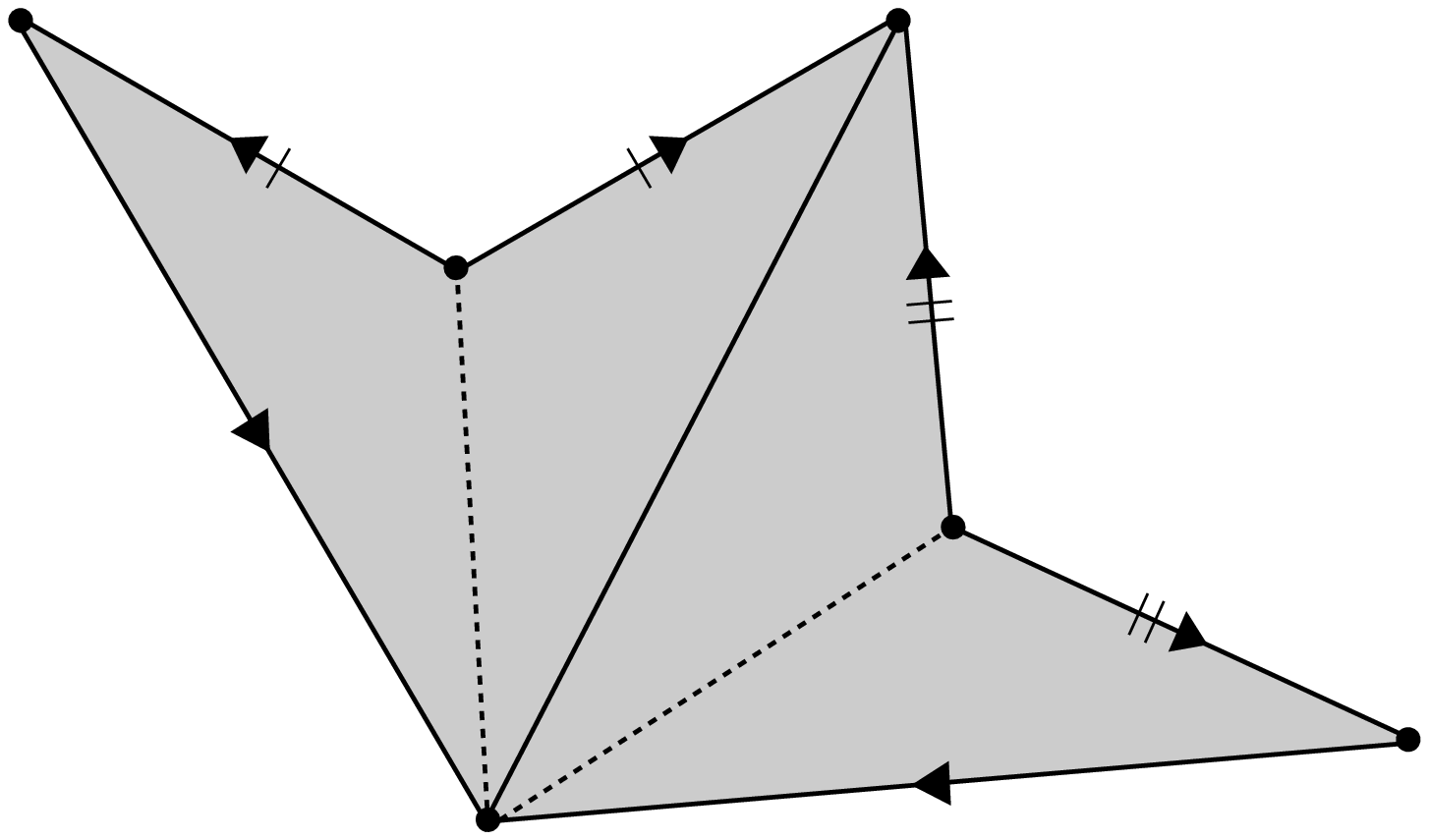}}
\hspace*{0.1\textwidth}
\labellist
\small\hair 2pt
\pinlabel $c'$ [b] at -119 622
\pinlabel $c''$ [b] at 145 622
\pinlabel $a$ [b] at 9 554
\pinlabel $d'$ [t] at 17 388
\pinlabel $d''$ [t] at 298 399
\pinlabel $b$ [t] at 155 472
\pinlabel $l_1$ [t] at 54 553
\pinlabel $l_2$ [b] at 107 467
\pinlabel $S$ at 165 512
\vspace*{11pt}
\endlabellist
\subfloat[The polygon $Z_1$]{\label{cor-p2-(2)}\includegraphics[width=0.2\textheight]{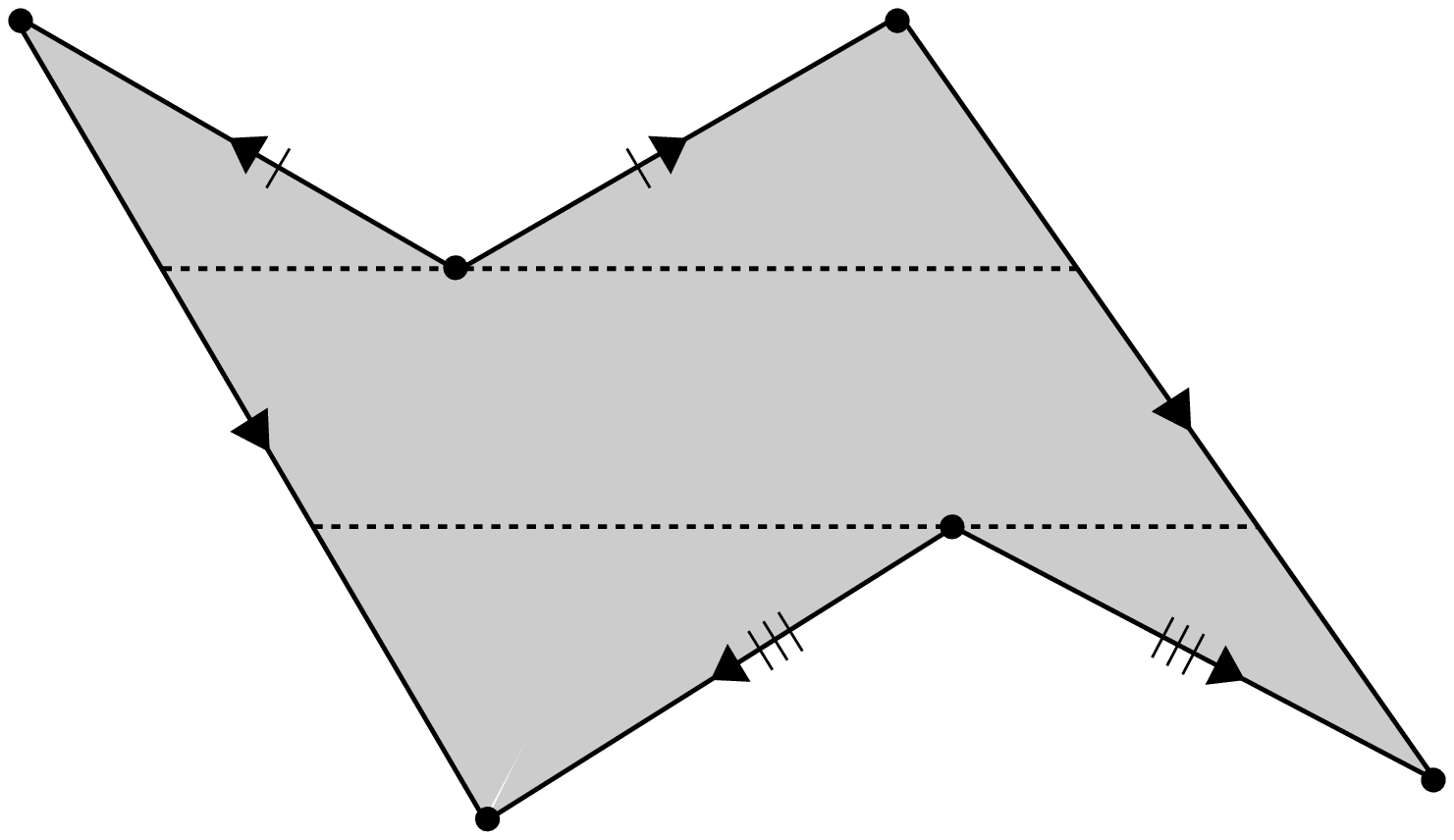}}
\caption{The polygons $Z$ and $Z_1$ in Case I.}
\end{figure}

($1$). Adding the two dashed segments from Figure \ref{cor-p2} divides the polygon $Z$ into four triangles, which gives rise to a triangulation of $M$. After looking at the identification pattern from Figure \ref{cor-p2}, it is obvious that $M$ embeds isometrically in the 3-dimensional Euclidean space as the boundary of a tetrahedron. Each face of this tetrahedron is an isosceles triangle.

($2$). Label the vertices of the polygon $Z$ as in Figure \ref{cor-p2}.  Cut $Z$ along the segment $bd$ and then glue the edges $bc''$ and $bc'''$, obtaining a new planar polygon $Z_1$. The hex sphere $M$ can be recovered by gluing the edges of $Z_1$ as shown in Figure \ref{cor-p2-(2)}. Label the vertices of $Z_1$ as in Figure \ref{cor-p2-(2)}. Let $l_1$ (respectively, $l_2$) be the segment in $Z_1$ that is parallel to the segment $c'c''$ and passes through the vertex $a$ (respectively, the vertex $b$). Let $S$ be the strip in $Z_1$ that is bounded by the segments $l_1$ and $l_2$. Since the edges $c'd'$ and $c''d''$ are parallel, then, after gluing the edges of $Z_1$ as shown in Figure \ref{cor-p2-(2)}, the strip $S$ gives rise to an annulus in $M$. The interior of this annulus satisfies the desired properties, and its core is a simple closed geodesic in $M$. 

($3$). Let $\gamma$ be the core of the annulus from ($2$) of the statement of the corollary. Let $S_a$ (respectively, $S_b$, $S_{c''}$ and $S_{d'}$) be the segment in $Z_1$ that is orthogonal to $\gamma$ and passes through the point $a$ (respectively, $b$, $c''$ and $d'$), see Figure \ref{cor-p2-(2)-bis-1}. Let $y$ (respectively, $z$) be the point where $\gamma$ intersects $S_{a}$ (respectively, $S_{d'}$). Perform a fractional Dehn twist along $\gamma$ until the points $y$ and $z$ coincide. Now it is easy to show that the hex sphere we obtain after the Dehn twist is the double of the perfect parallelogram $ad'bc''$.

\begin{figure}[ht!]
\labellist
\small\hair 3pt
\pinlabel $c'$ [b] at -119 622
\pinlabel $c''$ [b] at 140 622
\pinlabel $a$ [b] at 7 549
\pinlabel $d'$ [t] at 17 388
\pinlabel $d''$ [t] at 294 399
\pinlabel $b$ [t] at 153 474
\pinlabel $\gamma$ [b] at 80 505

\pinlabel $S_a$ [l] at 5 526
\pinlabel $S_b$ [r] at 157 489
\pinlabel $S_{c''}$ [r] at 142 563
\pinlabel $S_{d'}$ [l] at 16 449
\endlabellist
\centering
  \includegraphics[width=0.45\textwidth]{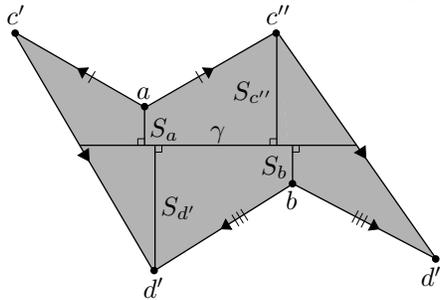}
 \caption{The hex sphere before the fractional Dehn twist along $\gamma$}
  \label{cor-p2-(2)-bis-1}
\end{figure}
The remaining cases are when both polygons $P_{A}$ and $P_{B}$ are either $3$- or $4$-special. In these cases, (1)-(3) can be proved using the same arguments as in Case I. The details are omitted.
\end{proof}

\section{A Picture is Worth a Thousand Words} \label{int-pts-Z-as-HS}

By Theorem \ref{thm-mod-spac-hex-spher}, the surface $\widehat{\mathcal{Z}}$ can be identified with the moduli space of hex spheres of unit area. Given a point in $\widehat{\mathcal{Z}}$, we would like to see the hex sphere this point corresponds to. This is summarized in Figure \ref{pts-in-Z-hat-as-HS}. 

\begin{figure}[ht!]
\vspace*{44pt}
\labellist
\small\hair 3pt
\pinlabel {$(\mathbb{R}_{\geq 0},0)$} [t] at -15 648
\pinlabel {$(\mathbb{R}_{\geq 0},0)$} [t] at 514 648
\pinlabel {$0$} [b] at -87 648
\pinlabel {$0$} [b] at 442 648
 \pinlabel \text{double of an} at 249 545
 \pinlabel \text{equilateral triangle} at 249 525
\pinlabel {$\overbrace{\hspace*{0.31\textwidth}}$} at 114 805
\pinlabel {$\overbrace{\hspace*{0.31\textwidth}}$} at 379 805
\pinlabel {$d_M(a,c) \leq d_M(a,d)$} at 114 825
\pinlabel {$d_M(a,c) \geq d_M(a,d)$} at 378 825
\pinlabel {$n=3$} [b] at 90 716
\pinlabel {$n=3$} [b] at 408 716
\pinlabel {$n=2$} [l] at 245 753
\pinlabel {doubles of} at 126 623
\pinlabel {parallelograms} at 126 603
\pinlabel {doubles of} at 374 623
\pinlabel {parallelograms} at 374 603

\endlabellist
\centering
  \includegraphics[width=0.8\textwidth]{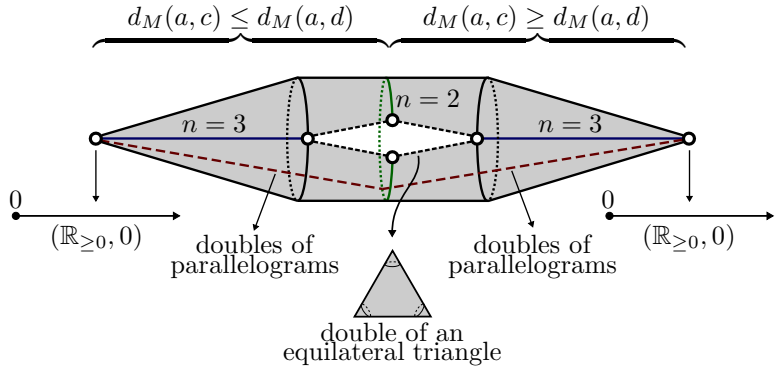}
\vspace*{11pt}
  \caption{Interpreting points in $\widehat{\mathcal{Z}}$ as hex spheres.}
  \label{pts-in-Z-hat-as-HS}
\end{figure}

The surface $\widehat{\mathcal{Z}}$ can be split into the halves $\mathcal{Z}^{+}$ and $\mathcal{Z}^{-}$. The left half $\mathcal{Z}^{+}$ (respectively, $\mathcal{Z}^{-}$) parametrizes the hex spheres $M$ for which $d_M(a,c) \leq d_M(a,d)$ (respectively, $d_M(a,c) \geq d_M(a,d)$). The regions $\mathcal{Z}^{+}$ and $\mathcal{Z}^{-}$ intersect in the open arc labeled as $n=2$ in Figure \ref{pts-in-Z-hat-as-HS}. This arc is defined by the relations $\varphi=0$ and $0<\alpha<\pi/3$ and its points correspond to hex spheres whose Voronoi polygons are $2$-special. These hex spheres are those that embed isometrically in the 3-dimensional Euclidean space as the boundary of a tetrahedron, all of whose faces are isosceles triangles.

There are two arcs in Figure \ref{pts-in-Z-hat-as-HS} labeled as $n=3$: one in $\mathcal{Z}^{+}$ and the other in $\mathcal{Z}^{-}$. These arcs are defined by the relations $\pi/6<\varphi<2\pi/3$ and $\alpha=\varphi-\pi/6$ and their points correspond to hex spheres whose Voronoi polygons are $3$-special. Each of these hex spheres is the double a Euclidean trapezoid with interior angles $\pi/3$, $\pi/3$, $2\pi/3$ and $2\pi/3$.

Points in $\widehat{\mathcal{Z}}$ that are \emph{not} in one of the arcs labeled as either $n=2$ or $n=3$ represent hex spheres whose Voronoi polygons are $4$-special. These form the generic type of hex spheres and, among them, there are some which are especially symmetric: the doubles of perfect parallelograms. The latter are parametrized by the dashed segment from Figure \ref{pts-in-Z-hat-as-HS}  that traverse the back of $\widehat{\mathcal{Z}}$ and whose middle point intersects the arc labeled $n=2$. In terms of the parameters $\varphi$ and $\alpha$, a generic hex sphere is defined by the inequalities $0<\varphi<2\pi/3$, $0<\alpha<\pi/2$ and $\varphi-\pi/6<\alpha<\varphi+\pi/3$. This hex sphere is the double of a perfect parallelogram if, additionally, $\alpha=\varphi/2+\pi/6$.

Recall that $\widehat{\mathcal{Z}}$ is homeomorphic to a sphere with three holes (not including the boundary of the holes). The leftmost and rightmost holes of $\widehat{\mathcal{Z}}$  will be referred to as \emph{small} holes, while the hole in the middle of $\widehat{\mathcal{Z}}$ will be referred to as the \emph{big} hole (see Figure \ref{pts-in-Z-hat-as-HS}).

Each hole in $\widehat{\mathcal{Z}}$ represents a \emph{degenerate} hex sphere, by which we mean a proper metric space that is the Gromov-Hausdorff limit of a sequence of genuine hex spheres. We consider only proper metric spaces because the Gromov-Hausdorff limits are unique when restricted to these spaces (see Corollary $6.11$ from \cite{Cooper}).

The big hole of $\widehat{\mathcal{Z}}$ represents the orbifold $S^2(2\pi/3,2\pi/3,2\pi/3)$, which is the double of an Euclidean equilateral triangle. This is so because every point in the boundary of the big hole is parametrized by a pair $(\varphi,\alpha) \in \mathbb{R}^2$ such that there is a sequence of points $(\varphi_n,\alpha_n) \in \widehat{\mathcal{Z}}$ converging to $(\varphi,\alpha)$ with
\begin{equation} \label{eqn-param-big-holes}
  \text{either } \alpha_n \to 0 \text{ or } \alpha_n \to \varphi+\pi/3
\end{equation}

Let $M_n$ be the hex sphere of unit area with angle parameters $\varphi_n$, $\alpha_n$, and let $a_n$, $b_n$ be the two cone points in $M_n$ of angle $4\pi/3$.  By \eqref{eqn-param-big-holes}, $d_{M_n}(a_n,b_n) \to 0$, and the degenerate hex sphere $S^2(2\pi/3,2\pi/3,2\pi/3)$ corresponds to the limiting case when the cone points $a_n$ and $b_n$ merge.

The small holes of $\widehat{\mathcal{Z}}$ represent non-compact degenerate hex spheres that experience  a dimensional collapsing. More precisely, each of these degenerate hex spheres is 1-dimensional but it is the limit of a sequence of long and skinny (2-dimensional) hex spheres. [The limit is a pointed Gromov-Hausdorff limit because of the non-compactness of the limiting space.] This is explained in the next paragraph. 

Consider the small hole in $\mathcal{Z}^{+}$, which is parametrized by the pair $(\varphi,\alpha)=(2\pi/3,\pi/2) \in \mathbb{R}^2$. There is a sequence of points $(\varphi_n,\alpha_n) \in \mathcal{Z}^{+}$ such that  $\varphi_n \to \varphi=2\pi/3$  and  $\alpha_n \to \alpha=\pi/2$. Let $M_n$ be the hex sphere of unit area  with angle parameters $\varphi_n$, $\alpha_n$.  Let $a_n$, $b_n$ be the two cone points in $M_n$ of angle $4\pi/3$, and let $c_n$, $d_n$ be the two cone points in $M_n$ of angle $2\pi/3$. Then, $d_{M_n}(a_n,c_n) \to 0$ and $d_{M_n}(a_n,d_n) \to \infty$, which imply that $d_{M_n}(b_n,d_n) \to 0$ and $d_{M_n}(b_n,c_n) \to \infty$ (by Theorem 4.4 from \cite{vor-graphs}). These conditions force the hex sphere $M_n$ to become longer and skinnier (by Corollary \ref {main-cor} $(1)$). Furthermore, the pointed Gromov-Hausdorff limit of the sequence $(M_n,a_n)$ is the metric space $(\mathbb{R}_{\geq 0}, 0)$, where $\mathbb{R}_{\geq 0}$ denotes the non-negative reals equipped with usual metric.

In general, the pointed Gromov-Hausdorff limit of a sequence of metric spaces depends on the base point. Thus, the Gromov-Hausdorff limit of the sequence $(M_n,x_n)$ will depend on the choice of the $x_n$'s. For instance, we just saw that if $x_n=a_n$ for all $n$, then the limit of the sequence $(M_n,x_n)$ is $(\mathbb{R}_{\geq 0}, 0)$. The reader can check that we obtain the same limit if, for each $n$, $x_n$ is any other of the cone points of $M_n$. However, if  $x_n$ is a point in the Voronoi graph of $M_n$ for each $n$, then the limit of the sequence $(M_n,x_n)$ is $(\mathbb{R}, 0)$. In a sense, it is more natural to take a cone point of $M_n$ as its base point, since cone points are the only ``distinguished'' points in a singular Euclidean surface. This would mean that $(\mathbb{R}_{\geq 0}, 0)$ is the ``natural'' limit of the sequence $(M_n,x_n)$, and this is why we included it in Figure \ref{pts-in-Z-hat-as-HS}.

The small hole in $\mathcal{Z}^{-}$ can be analyzed using the same arguments from the previous paragraphs.

\bibliographystyle{alpha}
\bibliography{biblio}

\end{document}